\documentclass[12pt]{article}

\usepackage[margin=1in]{geometry}
\usepackage{amsmath,amsthm,amssymb}
\usepackage{enumerate}
\usepackage{xfrac}
\usepackage{accents} % use if Aball and Bball are defined below
\usepackage[colorlinks=true,hidelinks]{hyperref}
    \AtBeginDocument{}
\usepackage{tikz} % for tikz pictures
\usetikzlibrary{arrows}
\usetikzlibrary{patterns}
\usepackage{pgfplots} % for graphing functions in tikz
\usetikzlibrary{decorations.markings} % use for arrows in tikz
\usetikzlibrary{3d,calc} %3d plotting in tikz
\usepackage{pifont}
\DeclareFontFamily{OT1}{pzc}{}
  \DeclareFontShape{OT1}{pzc}{m}{it}{<-> s * [1.200] pzcmi7t}{}
  \DeclareMathAlphabet{\mathpzc}{OT1}{pzc}{m}{it}
\usepackage{fancyhdr}
\newcommand{\N}{\mathbb{N}}
\newcommand{\Z}{\mathbb{Z}}
\newcommand{\Q}{\mathbb{Q}}

\newcommand{\C}{\mathbb{C}}
\newcommand{\F}{\mathbb{F}}

\renewcommand{\P}{\mathbb{P}}

\renewcommand{\S}{\mathcal{S}}

\newcommand{\D}{\mathcal{D}}

\DeclareFontFamily{U}{wncy}{}
    \DeclareFontShape{U}{wncy}{m}{n}{<->wncyr10}{}
    \DeclareSymbolFont{mcy}{U}{wncy}{m}{n}
    \DeclareMathSymbol{\Sha}{\mathord}{mcy}{"58} 

\newcommand{\LIN}{\ensuremath{\operatorname{LIN}}}
\newcommand{\stackspace}{3}
\newcommand{\stack}[2][1cm]{\;\tikz[baseline, yshift=.65ex]%
    {\foreach \k [evaluate=\k as \r using (.5*#2+.5-\k)*\stackspace] in {1,...,#2}{%
    \ifodd\k{\draw[<-](0,\r pt)--(#1,\r pt);}%
    \else{\draw[->](0,\r pt)--(#1,\r pt);}\fi
    }}\;}

\newcommand{\Spec}{\ensuremath{\operatorname{Spec}}}

\newcommand{\Aut}{\ensuremath{\operatorname{Aut}}}

\newcommand{\orb}{\mathcal{O}}

\newcommand{\mor}{\ensuremath{\operatorname{mor}}}
\newcommand*\cat[1]{{\tt #1}}

\newcommand{\Hom}{\ensuremath{\operatorname{Hom}}}

\newcommand{\Fun}{\ensuremath{\operatorname{Fun}}}

\newcommand{\Frob}{\ensuremath{\operatorname{Frob}}}

\newcommand{\A}{\mathbb{A}}

\newcommand{\eff}{\ensuremath{\operatorname{eff}}}

\renewcommand{\epsilon}{\varepsilon}
\newcommand{\re}{\text{Re}}

\newcommand{\legen}[2]{\left (\frac{#1}{#2}\right )}

\newtheorem{thm}{Theorem}[section]
\newtheorem{prop}[thm]{Proposition}

\newtheorem{lem}[thm]{Lemma}
\newtheorem{defn}[thm]{Definition}
\newtheorem{cor}[thm]{Corollary}

\newtheorem{rem}[thm]{Remark}
  \let\oldrem\rem
  \renewcommand{\rem}{\oldrem\normalfont}

\newtheorem{ex}[thm]{Example}
  \let\oldex\ex
  \renewcommand{\ex}{\oldex\normalfont}

\begin{document}

%\numberwithin{figure}{section}
%\let\stdsection\section
%\renewcommand\section{\newpage\stdsection}

%\interfootnotelinepenalty=10000

%-------------------------------------------------------------------------------------------------------------------------------------------------------------------------------------

\title{A Primer on Zeta Functions and Decomposition Spaces}
\author{Andrew Kobin\footnote{Emory University, Atlanta, USA, \; {\it email:} \url{ajkobin@emory.edu} \; The author is partially supported by the American Mathematical Society and the Simons Foundation.}}

\setcounter{footnote}{1}
\maketitle

%-------------------------------------------------------------------------------------------------------------------------------------------------------------------------------------

\rhead{\thepage}
\cfoot{}

\begin{abstract}
Many examples of zeta functions in number theory, combinatorics and algebraic geometry are special cases of a construction in homotopy theory known as a decomposition space. This article aims to introduce readers to the relevant concepts in homotopy theory and lays some foundations for future applications of decomposition spaces in the theory of zeta and $L$-functions. 
\end{abstract}

%-------------------------------------------------------------------------------------------------------------------------------------------------------------------------------------

\section{Introduction}

\vspace{0.1in}
This expository article is aimed at introducing number theorists to the theory of decomposition spaces in homotopy theory. Briefly, a decomposition space is a certain simplicial space that admits an abstract notion of an incidence algebra and, in particular, an abstract zeta function. To motivate the connections to number theory, we show that most classical notions of zeta functions in number theory and algebraic geometry are special cases of the construction using decomposition spaces. These examples suggest a wider application of decomposition spaces in number theory and algebraic geometry, which we intend to explore in future work. 

The paper is organized as follows. In Section~\ref{sec:zetafunctions}, we situate the Riemann zeta function, the Dedekind zeta function of a number field and the Hasse--Weil zeta function of a variety over a finite field in an algebra of arithmetic functions appropriate to each setting. This allows us to discuss some formal algebraic properties of these zeta functions in terms of a common product structure. We also discuss $L$-functions from this perspective in Section~\ref{sec:Lfun}. The heart of the paper is in Section~\ref{sec:incidencecoalgs}, where we describe the formalism of {\it decomposition spaces}, as developed by G\'{a}lvez-Carillo, Kock and Tonks \cite{gkt1,gkt2,gkt3}. This includes a more abstract discussion of the zeta functions in Section~\ref{sec:zetafunctions}. Finally, in Section~\ref{sec:future} we outline several future applications of the theory to important problems in number theory and algebraic geometry. These applications will likely require the full theory of decomposition spaces, not just the theory attached to posets. 

Some of the ideas for this article came from a reading course on $2$-Segal spaces organized by Julie Bergner at the University of Virginia in spring 2020, which resulted in the survey \cite{abd}. The author would like to thank the participants of this course, and in particular Bergner, Matt Feller and Bogdan Krstic, for enlightening discussions on these and related topics. The present work also owes much to a similar article by Joachim Kock \cite{koc}, which describes Riemann's zeta function from the perspective taken here. Kock notes that the techniques for posets go back at least to Stanley \cite{sta}. Finally, the author thanks Karen Acquista, Jon Aycock, Changho Han and Alicia Lamarche for comments on an early draft, and the referee for helpful suggestions on the final structure of the article.

%-------------------------------------------------------------------------------------------------------------------------------------------------------------------------------------

\section{Classical Zeta Functions}
\label{sec:zetafunctions}

We begin our story with one of the most famous objects in mathematics: the Riemann zeta function $\zeta_{\Q}(s)$. In the spirit of the homotopy-theoretic angle of this article, we aim to highlight certain aspects of $\zeta_{\Q}(s)$ which are ripe for generalization. 

\subsection{The Riemann Zeta Function}
\label{sec:riemannzeta}

In his landmark 1859 manuscript ``On the number of primes less than a given magnitude'', Riemann laid the foundation for modern analytic number theory with his extensive study of the function 
$$
\zeta(s) = \sum_{n = 1}^{\infty} \frac{1}{n^{s}}. 
$$
(To reserve the symbol $\zeta$ for later use, we will denote Riemann's zeta function by $\zeta_{\Q}(s)$.) The series $\zeta_{\Q}(s)$ converges for all \emph{complex numbers} $s$ with $\re(s) > 1$ and it conceals deep information about the nature of prime numbers within its analytic structure, e.g.~through its functional equation and nontrivial zeroes. One of the fundamental algebraic properties of $\zeta_{\Q}(s)$ is the product formula
$$
\zeta_{\Q}(s) = \prod_{p\text{ prime}} \left (1 - \frac{1}{p^{s}}\right )^{-1}. 
$$
Classically, the product formula is a consequence of unique factorization in $\Z$, which we reinterpret in Section~\ref{sec:posets} in the language of incidence algebras. 

Consider any function $f : \N\rightarrow\C$, which we will refer to as an {\it arithmetic function}\footnote{``Arithmetic function'' often refers to weakly multiplicative functions only. However, we do not restrict to such functions.} To $f$, we associate a Dirichlet series 
$$
F(s) = \sum_{n = 1}^{\infty} \frac{f(n)}{n^{s}}, \quad s\in\C. 
$$
For example, the Riemann zeta function is the Dirichlet series for the function $\zeta : \N\rightarrow\C$ defined by $\zeta(n) = 1$ for all $n$. We will see that other zeta functions can be constructed in a similar way. 

\begin{ex}
\label{ex:classicalMobfunction}
The {\it M\"{o}bius function} $\mu : \N\rightarrow\C$ is defined by 
$$
\mu(n) = \begin{cases}
  0, & p^{2}\mid n \text{ for some prime } p\\
  (-1)^{r}, & n = p_{1}\cdots p_{r} \text{ for distinct primes } p_{i}. 
\end{cases}
$$
We will write its corresponding Dirichlet series by 
$$
\mu_{\Q}(s) = \sum_{n = 1}^{\infty} \frac{\mu(n)}{n^{s}}. 
$$
The M\"{o}bius function can be characterized by the property that $\zeta_{\Q}(s) = \mu_{\Q}(s)^{-1}$. That is, 
$$
\sum_{n = 1}^{\infty} \frac{1}{n^{s}} = \left (\sum_{n = 1}^{\infty} \frac{\mu(n)}{n^{s}}\right )^{-1}. 
$$
There are classical proofs of this relation, but we will deduce it in a moment from the abstract properties of Dirichlet convolution. From now on, we ignore any considerations of convergence in our Dirichlet series. 
\end{ex}

For two arithmetic functions $f,g : \N\rightarrow\C$, their Dirichlet convolution is the function $f*g : \N\rightarrow\C$ defined by 
$$
(f*g)(n) = \sum_{ij = n} f(i)g(j)
$$
for all $n$. An equivalent definition is that Dirichlet convolution is the unique product on arithmetic functions corresponding to multiplication of their Dirichlet series: 

\begin{lem}
\label{lem:dirichletconv}
If $f,g : \N\rightarrow\C$ have Dirichlet series $F(s) = \sum_{n = 1}^{\infty} \frac{f(n)}{n^{s}}$ and $G(s) = \sum_{n = 1}^{\infty} \frac{g(n)}{n^{s}}$, respectively, then 
$$
F(s)G(s) = \sum_{n = 1}^{\infty} \frac{(f*g)(n)}{n^{s}}. 
$$
\end{lem}

The following properties of Dirichlet convolution are standard. 

\begin{prop}[{\cite[Sec.~3.6 - 3.7]{sta}}]
\label{prop:classicalMob}
Let $A = \{f : \N\rightarrow\C\}$ be the complex vector space of arithmetic functions. Then 
\begin{enumerate}[\quad (1)]
  \item $A$ is a commutative $\C$-algebra via Dirichlet convolution. 
  \item The function 
\begin{align*}
  \delta : \N &\longrightarrow \C\\
    n &\longmapsto \begin{cases} 1, & n = 1\\ 0, & n > 1\end{cases}
\end{align*}
  is a unit for convolution, making $A$ a unital $\C$-algebra. 
  \item $\mu*\zeta = \delta = \zeta*\mu$. 
\end{enumerate}
\end{prop}

In particular, (3) implies the formula on Dirichlet series: $\zeta_{\Q}(s) = \mu_{\Q}(s)^{-1}$. 

\begin{cor}[M\"{o}bius Inversion]
\label{cor:classicalMob}
For any $f,g\in A$, if $f = g*\zeta$ then $g = f*\mu$. That is, 
$$
\text{if}\quad f(n) = \sum_{d\mid n} g(d) \quad\text{then}\quad g(n) = \sum_{d\mid n} f(d)\mu\left (\frac{n}{d}\right ). 
$$
\end{cor}

\subsection{Dedekind Zeta Functions}
\label{sec:dedekindzeta}

The classical story of arithmetic functions on $\N$ and Dirichlet series generalizes in several directions. First, let $K/\Q$ be a number field with ring of integers $\orb_{K}$. Then $\orb_{K}$ has unique factorization of ideals. Let $I_{K}^{+}$ denote the semigroup of (nonzero) ideals $\frak{a}\subset\orb_{K}$ (the $+$ notation is to distinguish this from the group of fractional $\orb_{K}$-ideals $I_{K}$). The relation $\frak{a}\mid\frak{b}$ if and only if $\frak{b}\subseteq\frak{a}$ endows $I_{K}^{+}$ with the structure of a poset. 

Let $N = N_{K/\Q} : I_{K}^{+}\rightarrow\N,N(\frak{a}) = [\orb_{K} : \frak{a}]$, be the ideal norm. Then the Dedekind zeta function of $K/\Q$ can be defined by 
$$
\zeta_{K}(s) = \sum_{\frak{a}\in I_{K}^{+}} \frac{1}{N(\frak{a})^{s}} \quad\text{for } \re(s) > 1. 
$$
As with $\zeta_{\Q}(s)$, $\zeta_{K}(s)$ has a product formula: 
$$
\zeta_{K}(s) = \prod_{\frak{p}\in\Spec\orb_{K}} \left (1 - \frac{1}{N(\frak{p})^{s}}\right )^{-1}
$$
This follows from the unique factorization of prime ideals in $\orb_{K}$, which we will reinterpret in Section~\ref{sec:posets}. 

Mimicking our description of arithmetic functions in Section~\ref{sec:zetafunctions}, let us call any function $f : I_{K}^{+}\rightarrow\C$ an {\it arithmetic function} over $K$. To such a function we associate a Dirichlet series 
$$
F(s) = \sum_{\frak{a}\in I_{K}^{+}} \frac{f(\frak{a})}{N(\frak{a})^{s}}, \quad s\in\C. 
$$
Then the Dedekind zeta function is the Dirichlet series for the arithmetic function $\zeta : I_{K}^{+}\rightarrow\C$ defined by $\zeta(\frak{a}) = 1$ for all $\frak{a}\in I_{K}^{+}$. 

\begin{ex}
\label{ex:DedekindMob}
The M\"{o}bius function $\mu : I_{K}^{+}\rightarrow\C$ is the function 
$$
\mu(\frak{a}) = \begin{cases}
  0, & \frak{p}^{2}\mid\frak{a} \text{ for some prime ideal } \frak{p}\\
  (-1)^{r}, & \frak{a} = \frak{p}_{1}\cdots\frak{p}_{r} \text{ for distinct prime ideals } \frak{p}_{i}. 
\end{cases}
$$
We will write its corresponding Dirichlet series by 
$$
\mu_{K}(s) = \sum_{\frak{a}\in I_{K}^{+}} \frac{\mu(\frak{a})}{n^{s}}. 
$$
Then $\zeta_{K}(s) = \mu_{K}(s)^{-1}$, that is, 
$$
\sum_{\frak{a}\in I_{K}^{+}} \frac{1}{N(\frak{a})^{s}} = \left (\sum_{\frak{a}\in I_{K}^{+}} \frac{\mu(\frak{a})}{N(\frak{a})^{s}}\right )^{-1}. 
$$
\end{ex}

As we did for the Riemann zeta function, we can deduce this formula using a convolution product. For two arithmetic functions $f,g : I_{K}^{+}\rightarrow\C$, their convolution is the function $f*g : I_{K}^{+}\rightarrow\C$ defined by 
$$
(f*g)(\frak{a}) = \sum_{\frak{bc} = \frak{a}} f(\frak{b})g(\frak{c})
$$
for all $\frak{a}\subset\orb_{K}$. An equivalent characterization of $f*g$ is given by: 

\begin{lem}
\label{lem:Dedekinddirichletconv}
If $f,g : I_{K}^{+}\rightarrow\C$ have Dirichlet series $F(s) = \sum_{\frak{a}} \frac{f(\frak{a})}{N(\frak{a})^{s}}$ and $G(s) = \sum_{\frak{a}} \frac{g(\frak{a})}{N(\frak{a})^{s}}$, respectively, then 
$$
F(s)G(s) = \sum_{\frak{a}\in I_{K}^{+}} \frac{(f*g)(\frak{a})}{N(\frak{a})^{s}}. 
$$
\end{lem}

Then the analogue of Proposition~\ref{prop:classicalMob} is: 

\begin{prop}
\label{prop:DedekindMob}
Let $A_{K} = \{f : I_{K}^{+}\rightarrow\C\}$ be the complex vector space of functions on the semigroup $I_{K}^{+}$. Then 
\begin{enumerate}[\quad (1)]
  \item $A_{K}$ is a commutative $\C$-algebra via Dirichlet convolution. 
  \item The function 
\begin{align*}
  \delta : I_{K}^{+} &\longrightarrow \C\\
    \frak{a} &\longmapsto \begin{cases} 1, & \frak{a} = (1)\\ 0, & \frak{a}\not = (1)\end{cases}
\end{align*}
  is a unit for convolution, making $A_{K}$ a unital $\C$-algebra. 
  \item In $A_{K}$, we have $\mu*\zeta = \delta = \zeta*\mu$. 
\end{enumerate}
\end{prop}

In particular, (3) implies $\zeta_{K}(s) = \mu_{K}(s)^{-1}$. 

\begin{cor}[M\"{o}bius Inversion for Number Fields]
\label{cor:DedekindMob}
For any $f,g\in A_{K}$, if $f = g*\zeta$ then $g = f*\mu$. That is, 
$$
\text{if}\quad f(\frak{a}) = \sum_{\frak{d}\mid\frak{a}} g(\frak{d}) \quad\text{then}\quad g(\frak{a}) = \sum_{\frak{d}\mid\frak{a}} f(\frak{d})\mu(\frak{ad}^{-1})
$$
where $\frak{d}^{-1}$ denotes the inverse of $\frak{d}$ in the group of fractional ideals $I_{K}$. 
\end{cor}

In Section~\ref{sec:posets}, we identify $A_{K}$ with the reduced incidence algebra of the poset $(I_{K}^{+},\mid)$, and the above results will be subsumed by M\"{o}bius inversion for posets. 

\subsection{Hasse--Weil Zeta Functions}
\label{sec:HWzeta}

Another type of zeta function arises from algebraic geometry over finite fields, famously developed by Serre, Weil and others. Let $X$ be a variety over a finite field $k = \F_{q}$. For each $n\geq 1$, let $\#X(\F_{q^{n}})$ denote the number of points of $X$ defined over the unique field extension $\F_{q^{n}}$ of $\F_{q}$ of degree $n$. The Hasse--Weil zeta function of $X/\F_{q}$ is the formal power series 
$$
Z(X,t) = \exp\left [\sum_{n = 1}^{\infty} \frac{\#X(\F_{q^{n}})}{n}t^{n}\right ]
$$
where $\exp(t) = 1 + \sum_{n = 1}^{\infty} \frac{1}{n!}t^{n}$. This function is the subject of the Weil Conjectures, all of which have been settled: 
\begin{thm}[Weil Conjectures]
Let $X$ be a smooth, projective, geometrically connected variety over $\F_{q}$ and let $d = \dim X$. Then 
\begin{enumerate}[\quad (1)]
  \item (Weil 1948, Dwork 1960) $Z(X,t)$ is a rational function. 
  \item (Grothendieck 1965, Deligne 1974) $Z(X,t)$ satisfies the functional equation 
$$
Z(q^{-n}t^{-1}) = \epsilon q^{nE/2}t^{E}Z(t)
$$
  where $\epsilon = \pm 1$ and $E$ is the self-intersection number of the diagonal $X\hookrightarrow X\times X$. 
  \item (Deligne 1974) As a rational function, 
$$
Z(X,t) = \frac{P_{1}(t)P_{3}(t)\cdots P_{2d - 1}(t)}{P_{0}(t)P_{2}(t)\cdots P_{2d}(t)}
$$
with $P_{2d}(t) = 1 - q^{n}t$ and all other $P_{i}$ integer-valued polynomials in $t$ with roots $(\alpha_{ij})$ satisfying $|\alpha_{ij}| = q^{i/2}$. 
\end{enumerate}
\end{thm}

As in previous sections, we will focus on the algebraic properties of $Z(X,t)$ and ignore questions of convergence. One such property is the following product formula. 

\begin{prop}
\label{prop:HWzetaprod}
If $X$ is a variety over $\F_{q}$, then 
$$
Z(X,t) = \prod_{x\in |X|} (1 - t^{\deg(x)})^{-1}
$$
where $|X|$ denotes the set of closed points of $X$ and $\deg(x) = [k(x) : k]$ is the degree of a closed point. 
\end{prop}

\begin{proof}
For any $n\geq 1$, the set $X(\F_{q^{n}})$ may be identified with the set of closed points $x\in |X|$ with $\deg(x) \leq n$. Then 
$$
\#X(\F_{q^{n}}) = \sum_{d\mid n} d\cdot\#\{x\in |X| : \deg(x) = d\}
$$
and an easy manipulation shows that 
$$
\sum_{n = 1}^{\infty} \frac{\#X(\F_{q^{n}})}{n}t^{n} = -\log\left (\prod_{d = 1}^{\infty} (1 - t^{d})^{a_{d}}\right )
$$
where $a_{d} = \#\{x\in |X| : \deg(x) = d\}$ and $\log(\cdot)$ denotes the formal logarithmic power series $\log(1 + t) = \sum_{n = 1}^{\infty} \frac{(-1)^{n + 1}}{n}t^{n}$. Finally, applying $\exp$ gives to this equation yields 
$$
Z(X,t) = \prod_{d = 1}^{\infty} (1 - t^{d})^{-a_{d}} = \prod_{x\in |X|} (1 - t^{\deg(x)})^{-1}. 
$$
\end{proof}

Alternatively, $Z(X,t)$ can be written as the generating function of effective $0$-cycles on $X$, which we recall now. For a variety $X$ over an arbitrary field $k$, the group of $0$-cycles on $X$, denoted $Z_{0}(X)$, is the free abelian group generated by all closed points $x\in |X|$. Every $\alpha\in Z_{0}(X)$ can be written as a sum $\alpha = \sum_{x} a_{x}x$ over $x\in |X|$, where all but finitely many $a_{x}\in\Z$ are zero. An effective $0$-cycle is a $0$-cycle $\alpha = \sum_{x} a_{x}x$ such that $a_{x}\geq 0$ for all $x\in |X|$. The {\bf degree} of a $0$-cycle $\alpha = \sum_{x} a_{x}x$ is the integer $\deg(\alpha) = \sum_{x} a_{x}\deg(x)$. 

It follows from these definitions that for a variety $X$ over $k = \F_{q}$, the Hasse--Weil zeta function of $X$ can be written 
$$
Z(X,t) = \sum_{\alpha\in Z_{0}^{\eff}(X)} t^{\deg(\alpha)}
$$
where $Z_{0}^{\eff}(X)\subseteq Z_{0}(X)$ is the semigroup of effective $0$-cycles on $X$. 

\subsection{Zeta Functions of Arithmetic Schemes}
\label{sec:arithsch}

The number theoretic Riemann and Dedekind zeta functions and the algebro-geometric Hasse--Weil zeta functions come together in the following setting. Let $X$ be a scheme of finite type over $\Spec\Z$, also called an {\it arithmetic scheme}. The zeta function of an arithmetic scheme $X$ is the complex power series 
$$
\zeta_{X}(s) = \prod_{x\in |X|} \left (1 - \frac{1}{N(x)^{s}}\right )^{-1}
$$
where $N(x) = \#(\orb_{X,x}/\frak{m}_{x})$ is the cardinality of the finite residue field of a closed point $x$. 

\begin{rem}
In the literature, this is sometimes also called the {\it Hasse--Weil function} of an arithmetic scheme. Indeed, this was Hasse and Weil's original situation, which was later translated to varieties over finite fields. 
\end{rem}

\begin{prop}
\label{prop:prodformarithsch}
For an arithmetic scheme $X$ and a prime $p$, let $X_{p}$ denote the reduction mod $p$ of $X$, i.e.~$X_{p} = X\times_{\Spec\Z}\Spec\F_{p}$. Then 
$$
\zeta_{X}(s) = \prod_{p \text{ prime}} Z(X_{p},p^{-s})
$$
where $Z(X_{p},t)$ is the Hasse--Weil zeta function of the $\F_{p}$-variety $X_{p}$. 
\end{prop}

\begin{proof}
Use Proposition~\ref{prop:HWzetaprod}. 
\end{proof}

\begin{ex}
For $X = \Spec\Z$ itself, we have $X_{p} = \Spec\F_{p}$ for all primes $p$. Moreover, 
$$
Z(\Spec\F_{p},t) = \exp\left [\sum_{n = 1}^{\infty} \frac{1}{n}t^{n}\right ] = \exp[-\log(1 - t)] = (1 - t)^{-1}. 
$$
Hence by Proposition~\ref{prop:prodformarithsch}, 
$$
\zeta_{\Spec\Z}(s) = \prod_{p\text{ prime}} (1 - p^{-s})^{-1} = \zeta_{\Q}(s), 
$$
the Riemann zeta function. It would not be unreasonable to view this as the \emph{definition} of the Riemann zeta function: it is the zeta function attached to the terminal object $\Spec\Z$ in the category of schemes. 
\end{ex}

\begin{ex}
Let $K/\Q$ be a number field with ring of integers $\orb_{K}$. Then $X = \Spec\orb_{K}$ has reductions $X_{p} = \Spec(\orb_{K}/p\orb_{K}) \cong \coprod_{i = 1}^{r} \Spec(\orb_{K}/\frak{p}_{i}^{e_{i}})$ where $p\orb_{K} = \prod_{i = 1}^{r} \frak{p}_{i}^{e_{i}}$ for distinct prime ideals $\frak{p}_{i}\subset\orb_{K}$ and integers $e_{i}\geq 1$. Let $f_{i} = f(\frak{p}_{i}\mid p) = \dim_{\F_{p}}(\orb_{K}/\frak{p}_{i})$ be the inertia degree of $\frak{p}_{i}$ over $p$. As a result, 
$$
\#X_{p}(\F_{p^{n}}) = \sum_{i = 1}^{r} \#\Spec(\orb_{K}/\frak{p}_{i}^{e_{i}})(\F_{p^{n}}) = \sum_{i = 1}^{r} \#\Hom(\orb_{K}/\frak{p}_{i}^{e_{i}},\F_{p^{n}}). 
$$
For each $\frak{p}_{i}$ lying over $p$, there is a $\F_{p}$-linear map $\orb_{K}/\frak{p}_{i}^{e_{i}}\rightarrow\F_{p^{n}}$ exactly when $f_{i}\mid n$, and in this case $\#\Hom(\orb_{K}/\frak{p}_{i}^{e_{i}},\F_{p^{n}}) = f_{i}$. Thus the local factor at $p$ is 
\begin{align*}
  Z(X_{p},t) &= \exp\left [\sum_{n = 1}^{\infty}\sum_{i = 1}^{r} \frac{\#\Hom(\orb_{K}/\frak{p}_{i}^{e_{i}},\F_{p^{n}})}{n}t^{n}\right ] = \prod_{i = 1}^{r} \exp\left [\sum_{n = 1}^{\infty} \frac{\#\Hom(\orb_{K}/\frak{p}_{i}^{e_{i}},\F_{p^{n}})}{n}t^{n}\right ]\\
    &= \prod_{i = 1}^{r} \exp\left [\sum_{n = 1}^{\infty} \frac{1}{n}t^{f_{i}n}\right ] = \prod_{i = 1}^{r} \exp[-\log(1 - t^{f_{i}})] = \prod_{i = 1}^{r} (1 - t^{f_{i}})^{-1}
\end{align*}
Evaluating at $t = p^{-s}$ gives  
$$
Z(X_{p},p^{-s}) = \prod_{i = 1}^{r} (1 - (p^{-s})^{f_{i}})^{-1} = \prod_{i = 1}^{r} (1 - p^{-f_{i}s})^{-1}. 
$$
Recall that for each $\frak{p}_{i}$ lying over $p$, $N(\frak{p}_{i}) = p^{f_{i}}$. Putting the factors together using Proposition~\ref{prop:prodformarithsch}, we have 
$$
\zeta_{\Spec\orb_{K}}(s) = \prod_{p\text{ prime}} \prod_{i = 1}^{r} (1 - p^{-f_{i}s})^{-1} =  \prod_{\frak{p}\in\Spec\orb_{K}} (1 - N(\frak{p})^{-s})^{-1}
$$
which is precisely the Dedekind zeta function for $K/\Q$. Said another way, $\zeta_{K}(s)$ is the zeta function attached to the terminal object in the category of $\orb_{K}$-schemes. 
\end{ex}

\begin{ex}
\label{ex:globalzetaAn}
Let $X = \A_{\Z}^{n}$ be affine $n$-space over the integers. Then for each prime $p$, $X_{p} = \A_{\F_{p}}^{n}$ which has $\#\A^{n}(\F_{p^{k}}) = p^{nk}$ points over any $\F_{p^{k}}$. Thus 
$$
Z(X_{p},t) = \exp\left [\sum_{k = 1}^{\infty} \frac{p^{nk}}{k}t^{k}\right ] = \exp\left [\sum_{k = 1}^{\infty} \frac{1}{k}(p^{n}t)^{k}\right ] = (1 - p^{n}t)^{-1}. 
$$
By Proposition~\ref{prop:prodformarithsch}, 
$$
\zeta_{\A_{\Z}^{n}}(s) = \prod_{p\text{ prime}} (1 - p^{n}p^{-s})^{-1} = \zeta_{\Q}(s - n). 
$$
\end{ex}

\begin{lem}
\label{lem:arithzetadecomp}
If $Z\hookrightarrow X$ is a closed arithmetic subscheme with complement $U = X\smallsetminus Z$, then $\zeta_{X}(s) = \zeta_{Z}(s)\zeta_{U}(s)$. 
\end{lem}

\begin{proof}
Use Proposition~\ref{prop:prodformarithsch} and the inclusion-exclusion formula for the point counts over each finite field, or see \cite[Rem.~6.32]{mus}. 
\end{proof}

\begin{ex}
For $X = \P_{\Z}^{1}$, we can decompose $\P_{\Z}^{1} = \A_{\Z}^{1}\cup\{\infty\}$ where the point $\infty$ is treated as a closed subscheme $\Spec\Z\hookrightarrow\P_{\Z}^{1}$. Then by Lemma~\ref{lem:arithzetadecomp}, 
$$
\zeta_{\P_{\Z}^{1}}(s) = \zeta_{\Q}(s - 1)\zeta_{\Q}(s). 
$$
More generally, 
$$
\zeta_{\P_{\Z}^{n}}(s) = \zeta_{\Q}(s - n)\cdots\zeta_{\Q}(s - 1)\zeta_{\Q}(s). 
$$
\end{ex}

In general, it is expected that $\zeta_{X}(s)$ has meromorphic continuation to $\C$ and satisfies a functional equation analogous to that of the Riemann and Dedekind zeta functions, at least when $X$ is nice (i.e.~regular and proper over $\Spec\Z$). Further, there are deep conjectures about the values, order of poles, etc.~of $\zeta_{X}(s)$ which we won't spell out here. Suffice it to say that these are some of the most sought-after results in all of number theory. For our purposes, arithmetic zeta functions are one motivation for developing a more general framework to handle different types of zeta functions simultaneously. 

\subsection{$L$-Functions}
\label{sec:Lfun}

Of course the focus in much of modern number theory is not just on zeta functions, but on $L$-functions. Fortunately, these fit into the same general framework described above. Historically, an $L$-function is either a generalization of the Riemann zeta function -- an Artin $L$-function -- or an analytic object with similar properties to an Artin $L$-function, such as the $L$-function of an automorphic form. A common feature of these different classes of $L$-functions is that they can be expressed as a Dirichlet series 
$$
L(s) = \sum_{n = 1}^{\infty} \frac{a_{n}}{n^{s}}. 
$$
Given such an expression, the $L$-function naturally corresponds to an element $L : n\mapsto a_{n}$ in the algebra $A$ of arithmetic functions from Section~\ref{sec:zetafunctions}. In this way, algebraic properties of $L$-functions may be encoded by the convolution product on $A$, cf.~\cite{ak1}. 

More explicitly, let $K/\Q$ be a number field with Galois group $G$ and fix a Galois module $V$ corresponding to some $n$-dimensional representation $\rho$ of $G$. The $L$-function of $V$ is defined as 
$$
L(V,s) = \prod_{\frak{p}} L_{\frak{p}}(V,s)
$$
where each prime factor $L_{p}(V,s)$ is the characteristic polynomial of the Frobenius at $p$ acting on the inertia-invariant vectors of $V$: 
$$
L_{\frak{p}}(V,s) = \frac{1}{\det(1 - N(\frak{p})^{-s}\Frob_{\frak{p}} \mid V^{I_{\frak{p}}})}
$$
where $I_{\frak{p}}$ is the inertia subgroup of $G$ for the prime $\frak{p}\subset\orb_{K}$. Grouping together the factors $L_{\frak{p}}(V,s)$ for all $\frak{p}$ lying over some fixed prime $p$, we get a Dirichlet series 
$$
L_{p}(V,s) := \prod_{\frak{p}\mid p} L_{\frak{p}}(V,s)
$$
for each rational prime $p$. Taking $V$ to be the trivial $1$-dimensional representation of $G$, it is immediate that each $\Frob_{\frak{p}}$ acts as the identity on $V^{I_{\frak{p}}}$, so the local factors of $L(V,s)$ are just the local factors of the Dedekind zeta function: 
$$
L_{\frak{p}}({\bf 1},s) = \frac{1}{1 - N(\frak{p})^{-s}}. 
$$
That is, $L({\bf 1},s) = \zeta_{K}(s)$. By class field theory, the $L$-function of any $1$-dimensional representation is in fact a Dirichlet $L$-function $L(\chi,s)$ for some Dirichlet character $\chi : \Z\rightarrow\C$. Already at this level, we can encode interesting algebraic relations among $L(\chi,s)$ in the convolution algebra $A$; for example, see \cite{ak1} for an objective proof of the formula 
$$
\zeta_{K}(s) = \zeta_{\Q}(s)L(\chi,s)
$$
for a quadratic extension $K/\Q$ with quadratic Dirichlet character $\chi$. 

If $L/K$ is an arbitrary extension of number fields with Galois group $G$, Artin $L$-functions $L(V,s)$ are constructed for any representation of $G$ in a similar fashion as above. When $V$ is a $1$-dimensional representation, class field theory once again prescribes a Hecke character $\chi : I_{K}\rightarrow\C$, where $I_{K}$ is the group of fractional ideals of $\orb_{K}$. More generally, the Langlands program predicts a correspondence between Artin $L$-functions and $L$-functions of automorphic representations of $G$. 

\begin{rem}
\label{rem:Lfun}
Artin $L$-functions are usually defined ``over $\Q$'', that is, as Dirichlet series which correspond to functions in the convolution algebra $A = A_{\Q}$. However, it is equally valid to interpret them as elements of the convolution algebra $A_{K}$ where $K$ is the ground field of the extension defining the $L$-functions. This slight shift in perspective allows us to unpack algebraic properties of $L(V,s)$ in $A_{K}$ directly, before passing to $A_{\Q}$ using the pushforward map described in Remark~\ref{rem:pushfwd}. 
\end{rem}

%-------------------------------------------------------------------------------------------------------------------------------------------------------------------------------------

\section{Decomposition Spaces and Incidence Coalgebras}
\label{sec:incidencecoalgs}

\renewcommand{\A}{\mathbb{A}}

In this section, we survey the theory of decomposition spaces, originating from the work of G\'{a}lvez-Carrillo, Kock and Tonks in \cite{gkt1,gkt2,gkt3}. This theory is developed in order to discuss incidence algebras in their most general form. In particular, the incidence algebra of a decomposition space has a canonical element called the {\it zeta functor} which generalizes most existing examples of zeta functions. To motivate the formalism of decomposition spaces, we first show how number theoretic and algebro-geometric zeta functions can be recovered from incidence algebras. 

\subsection{Motivating Examples}
\label{sec:posets}

The Riemann, Dedekind and Hasse--Weil zeta functions all arise from incidence algebras of \emph{posets}, which turn out to be simple examples of decomposition sets (defined in Section~\ref{sec:decompset}). Although the full power of decomposition spaces is not needed in this section, this will motivate the general theory, along with some future applications at the end of the article. 

Let $(\mathcal{P},\leq)$ be a poset and for any elements $x,y\in\mathcal{P}$, define the interval $[x,y]$ by 
$$
[x,y] = \{z\in\mathcal{P}\mid x\leq z\leq y\}. 
$$

\begin{ex}
For the poset $(\N,\leq)$, $[x,y]$ is the usual interval of integers between $x$ and $y$. More useful to us will be the poset $(\N_{0},\leq)$ of nonnegative integers ordered by succession, which is isomorphic to $(\N,\leq)$. 
\end{ex}

\begin{ex}
Let $(\N,\mid)$ denote the divisibility poset of the natural numbers. Then $[x,y] = \{d\in\N : x\mid d\mid y\}$. 
\end{ex}

\begin{ex}
For any number field $K/\Q$, the set of ideals of $\orb_{K}$ forms a poset $(I_{K}^{+},\mid)$ with intervals $[\frak{a},\frak{b}] = \{\frak{d} : \frak{a}\mid\frak{d}\mid\frak{b}\}$. 
\end{ex}

\begin{defn}
A poset $(\mathcal{P},\leq)$ is {\bf locally finite} if every interval $[x,y]$ in $\mathcal{P}$ is a finite set. 
\end{defn}

\begin{rem}
All of the above posets are locally finite. 
\end{rem}

Fix a field $k$. The following definitions go back at least to Stanley \cite{sta}, Rota \cite{rot} and their contemporaries. 

\begin{defn}
The {\bf incidence coalgebra} $C(\mathcal{P})$ of a locally finite poset $(\mathcal{P},\leq)$ over $k$ is the free $k$-vector space on the set of intervals $\{[x,y] : x,y\in\mathcal{P}\}$, together with the comultiplication and counit maps 
\begin{align*}
  \Gamma : C(\mathcal{P}) &\longrightarrow C(\mathcal{P})\otimes C(\mathcal{P})\\
    [x,y] &\longmapsto \sum_{z\in [x,y]} [x,z]\otimes [z,y]\\
  \delta : C(\mathcal{P}) &\longrightarrow k\\
    [x,y] &\longmapsto \delta_{xy} = \begin{cases} 1, & x = y\\ 0, & x\not = y\end{cases}. 
\end{align*}
\end{defn}

\begin{defn}
The {\bf incidence algebra} $I(\mathcal{P})$ of $(\mathcal{P},\leq)$ over $k$ is the $k$-vector space $I(\mathcal{P}) = \Hom_{k}(C(\mathcal{P}),k)$ together with the multiplication map 
\begin{align*}
  * : I(\mathcal{P})\otimes I(\mathcal{P}) &\longrightarrow I(\mathcal{P})\\
    \varphi\otimes\psi &\longmapsto \left (\varphi*\psi : [x,y] \mapsto \sum_{z\in [x,y]} \varphi([x,z])\psi([z,y])\right )
\end{align*}
and unit $1\mapsto \delta$, also written $\delta$ by abuse of notation. 
\end{defn}

By \cite[Thm.~7.4]{gkt1}, the incidence (co)algebra of a locally finite poset is (co)associative and (co)unital. It need not be (co)commutative in general. 

In any incidence algebra $I(\mathcal{P})$ for a locally finite poset $\mathcal{P}$, there is a distinguished element $\zeta : [x,y]\mapsto 1$, called the {\it zeta function} for $\mathcal{P}$. If $\zeta\in I(\mathcal{P})^{\times}$, we denote an inverse by $\mu = \zeta^{-1}$, called the {\it M\"{o}bius function} for $\mathcal{P}$. 

\begin{prop}[M\"{o}bius Inversion]
\label{prop:posetMob}
For any locally finite poset $(\mathcal{P},\leq)$, 
\begin{enumerate}[\quad (1)]
  \item $\mu = \zeta^{-1}$ exists and is defined recursively by 
$$
\mu : [x,y] \longmapsto \begin{cases}
  1, & x = y\\
  -\sum_{z\in [x,y)} \mu([x,z]), & x\not = y. 
\end{cases}
$$
  \item (Rota's Formula) For any $f,g\in I(\mathcal{P})$, if $f = g*\zeta$ then $g = f*\mu$. That is, 
$$
\text{if}\quad f([x,y]) = \sum_{z\in [x,y]} g([x,z]) \quad\text{then}\quad g([x,y]) = \sum_{z\in [x,y]} f([x,z])\mu([z,y]). 
$$
\end{enumerate}
\end{prop}

To recover classical versions such as Corollary~\ref{cor:classicalMob}, it is useful to pass to the {\it reduced incidence algebra} $\widetilde{I}(\mathcal{P})$, the subalgebra of $I(\mathcal{P})$ consisting of $\varphi$ that are constant on isomorphism classes of intervals (considered as abstract posets). Alternatively, $\widetilde{I}(\mathcal{P})$ is the dual of the coalgebra $\widetilde{C}(\mathcal{P})$ of isomorphism classes of intervals, cf.~\cite[2.5]{koc}. 

\begin{ex}
\label{ex:leqposet}
For the poset $(\N_{0},\leq)$ and $k = \C$, the M\"{o}bius function is 
$$
\mu([x,y]) = \begin{cases}
  1, & x = y\\
  -1, & y = x + 1\\
  0, &\text{ otherwise}. 
\end{cases}
$$
Here, every interval is isomorphic to one of the form $[0,y - x]$ for $y\geq x$. For an element $f\in\widetilde{I}(\N_{0},\leq)$ which is a priori a function on the intervals of $(\N_{0},\leq)$, we write $f(n) = f([0,n])$. In this case, M\"{o}bius inversion (applied to the reduced incidence algebra $\widetilde{I}(\N_{0},\leq)$) says that 
$$
f(n) = \sum_{i\leq n} g(i) \implies g(n) = f(n) - f(n - 1). 
$$
For the poset $(\N_{0},\leq)$, we can also interpret functions $f : \N\rightarrow\C$ in terms of their generating functions 
$$
F(z) = \sum_{n = 0}^{\infty} f(n)z^{n}. 
$$
Ignoring questions of convergence, we have 
$$
\sum_{n = 0}^{\infty} \zeta(n)z^{n} = \sum_{n = 0}^{\infty} z^{n} = \frac{1}{1 - z} \quad\text{and}\quad \sum_{n = 0}^{\infty} \mu(n)z^{n} = 1 - z
$$
as M\"{o}bius inversion predicts. 
\end{ex}

\begin{ex}
\label{ex:divposet}
For the poset $(\N,\mid)$ and $k = \C$, the M\"{o}bius function is precisely the classical $\mu$ from Example~\ref{ex:classicalMobfunction} and multiplication in $\widetilde{I}(\N,\mid)$ is Dirichlet convolution. That is, $\widetilde{I}(\N,\mid)$ is precisely the convolution algebra $A$ from Section~\ref{sec:riemannzeta}. Then Proposition~\ref{prop:classicalMob} and Corollary~\ref{cor:classicalMob} follow directly from Proposition~\ref{prop:posetMob}. For each prime $p$, consider the subposet of $p$th powers $(\{p^{k}\},\mid)\subseteq (\N,\mid)$. Then it is clear that $(\{p^{k}\},\mid)\cong (\N_{0},\leq)$ as posets, via $p^{k}\leftrightarrow k$. By the fundamental theorem of arithmetic, $(\N,\mid)$ decomposes as a restricted product of posets 
$$
(\N,\mid) \cong \prod_{p\text{ prime}}\!\!\!\!'\, (\{p^{k}\},\mid) \cong \prod_{p\text{ prime}}\!\!\!\!'\, (\N_{0},\leq). 
$$
By Example~\ref{ex:leqposet}, the M\"{o}bius function for each $(\{p^{k}\},\mid)$ can be written $\mu_{(\{p^{k}\},\mid)} = \delta - \delta_{p}$, where 
$$
\delta_{p}(k) = \begin{cases}
  1, & k = 1\\
  0, & k\not = 1. 
\end{cases}
$$
Then under the decomposition of $(\N,\mid)$ above, we have 
$$
\mu_{(\N,\mid)} = \bigotimes_{p\text{ prime}} \mu_{(\{p^{k}\},\mid)} = \bigotimes_{p\text{ prime}} (\delta - \delta_{p}). 
$$
In terms of Dirichlet functions, this means 
$$
\mu_{\Q}(s) = \sum_{n = 1}^{\infty} \frac{\mu(n)}{n^{s}} = \prod_{p\text{ prime}} \left (\sum_{n = 1}^{\infty} \frac{\delta(n)}{n^{s}} - \sum_{n = 1}^{\infty} \frac{\delta_{p}}{n^{s}}\right ) = \prod_{p\text{ prime}} \left (1 - \frac{1}{p^{s}}\right ). 
$$
Finally, since $\zeta_{\Q}(s) = \mu_{\Q}(s)^{-1}$, we recover Euler's product formula 
$$
\zeta_{\Q}(s) = \prod_{p\text{ prime}} \left (1 - \frac{1}{p^{s}}\right )^{-1}. 
$$
\end{ex}

\begin{ex}
More generally, for a number field $K/\Q$, consider the poset $(I_{K}^{+},\mid)$, which is locally finite by unique factorization of ideals in $\orb_{K}$. The M\"{o}bius function for this poset is the function $\mu$ defined in Example~\ref{ex:DedekindMob} and once again, multiplication in the reduced incidence algebra $\widetilde{I}(I_{K}^{+},\mid)$ is Dirichlet convolution, so we recover Proposition~\ref{prop:DedekindMob} and Corollary~\ref{cor:DedekindMob} from the general case in Proposition~\ref{prop:posetMob}. For each prime ideal $\frak{p}\in\Spec\orb_{K}$, the subposet $(\{\frak{p}^{k}\},\mid)\subseteq (I_{K}^{+},\mid)$ is isomorphic to $(\N_{0},\leq)$, so $(I_{K}^{+},\mid)$ decomposes as 
$$
(I_{K}^{+},\mid) \cong \prod_{\frak{p}\in\Spec\orb_{K}} (\{\frak{p}^{k}\},\mid) \cong \prod_{\frak{p}\in\Spec\orb_{K}} (\N_{0},\leq). 
$$
Then by Example~\ref{ex:leqposet}, the M\"{o}bius function for $(\{\frak{p}^{k}\},\mid)$ is $\mu_{(\{\frak{p}^{k}\},\mid)} = \delta - \delta_{\frak{p}}$, where 
$$
\delta_{\frak{p}}(\frak{a}) = \begin{cases}
  1, & \frak{a} = \frak{p}\\
  0, & \frak{a}\not = \frak{p}. 
\end{cases}
$$
Then the M\"{o}bius function of $(I_{K}^{+},\mid)$ decomposes as 
$$
\mu_{(I_{K}^{+},\mid)} = \bigotimes_{\frak{p}\in\Spec\orb_{K}} \mu_{(\{\frak{p}^{k}\},\mid)} = \bigotimes_{\frak{p}\in\Spec\orb_{K}} (\delta - \delta_{\frak{p}}). 
$$
Passing to Dirichlet functions, we get 
$$
\mu_{K}(s) = \sum_{\frak{a}\in I_{K}^{+}} \frac{\mu(\frak{a})}{N(\frak{a})^{s}} = \prod_{\frak{p}\in\Spec\orb_{K}} \left (\sum_{\frak{a}\in I_{K}^{+}} \frac{\delta(\frak{a})}{N(\frak{a})^{s}} - \sum_{\frak{a}\in I_{K}^{+}} \frac{\delta_{\frak{p}}(\frak{a})}{N(\frak{a})^{s}}\right ) = \prod_{\frak{p}\in\Spec\orb_{K}} \left (1 - \frac{1}{N(\frak{p})^{s}}\right ). 
$$
Therefore by M\"{o}bius inversion, 
$$
\zeta_{K}(s) = \mu_{K}(s)^{-1} = \prod_{\frak{p}\in\Spec\orb_{K}} \left (1 - \frac{1}{N(\frak{p})^{s}}\right )^{-1}. 
$$
\end{ex}

\begin{ex}
\label{ex:Mobinvvariety}
Let $X$ be a variety over $\F_{q}$ and consider the poset $(Z_{0}^{\eff}(X),\leq)$, where $\alpha \leq \beta$ if and only if $\alpha = \sum_{x} a_{x}x$, $\beta = \sum_{x} b_{x}x$ and $a_{x}\leq b_{x}$ for all $x\in |X|$. Then $(Z_{0}^{\eff}(X),\leq)$ is a locally finite poset with M\"{o}bius function 
$$
\mu : \alpha = \sum_{x} a_{x}x \longmapsto \begin{cases}
  1, & \alpha = 0\\
  0, & a_{x} > 1 \text{ for any } x\\
  (-1)^{r}, & \alpha = x_{1} + \ldots + x_{r} \text{ for distinct closed points } x_{i}. 
\end{cases}
$$
Let $A_{X}$ be the complex vector space of functions $f : Z_{0}^{\eff}(X)\rightarrow\C$ and define their convolution product by 
$$
(f*g)(\alpha) = \sum_{\beta + \gamma = \alpha} f(\beta)g(\gamma)
$$
where the sum is over all partitions of $\alpha$ into effective $0$-cycles, i.e.~$\beta$ and $\gamma$ such that $\beta + \gamma = \alpha$. 

\begin{prop}
\label{prop:Mobinv0cycles}
For a variety $X$ over $\F_{q}$, 
\begin{enumerate}[\quad (1)]
  \item $A_{X}$ is a commutative $\C$-algebra via the convolution product defined above. 
  \item The function 
\begin{align*}
  \delta : Z_{0}^{\eff}(X) &\longrightarrow \C\\
    \alpha &\longmapsto \begin{cases} 1, & \alpha = 0 \\ 0, & \alpha > 0\end{cases}
\end{align*}
  is a unit for convolution, making $A_{X}$ a unital $\C$-algebra. 
  \item In $A_{X}$, we have $\mu*\zeta = \delta = \zeta*\mu$ where $\zeta$ is the zeta function sending $\alpha\mapsto 1$ for all $\alpha$. 
  \item (M\"{o}bius Inversion for $0$-Cycles) For any $f,g\in A_{X}$, 
$$
\text{if}\quad f(\alpha) = \sum_{\beta\leq\alpha} g(\beta) \quad\text{then}\quad g(\alpha) = \sum_{\beta\leq\alpha} f(\beta)\mu(\alpha - \beta). 
$$
\end{enumerate}
\end{prop}

\begin{proof}
This is purely a consequence of Proposition~\ref{prop:posetMob}, after identifying $A_{X}$ with the reduced incidence algebra $\widetilde{I}(Z_{0}^{\eff}(X),\leq)$. 
\end{proof}

For any function $f\in A_{X}$, we can form a generating series 
$$
F(t) = \sum_{\alpha\in Z_{0}^{\eff}(X)} f(\alpha)t^{\deg(\alpha)},
$$
analogously to forming a Dirichlet series in the number field case. The following lemma is immediate. 

\begin{lem}
\label{lem:HWdirichletconv}
If $f,g\in A_{X}$ have generating series $F(t) = \sum_{\alpha} f(\alpha)t^{\deg(\alpha)}$ and $G(t) = \sum_{\alpha} g(\alpha)t^{\deg(\alpha)}$, respectively, then 
$$
F(t)G(t) = \sum_{\alpha} (f*g)(\alpha)t^{\deg(\alpha)}. 
$$
\end{lem}

As we saw in Section~\ref{sec:HWzeta}, the Hasse--Weil zeta function $Z(X,t)$ is the generating series for the abstract zeta function $\zeta\in A_{X}$. We will write the generating series for $\mu\in A_{X}$ by 
$$
M(X,t) = \sum_{\alpha\in Z_{0}^{\eff}(X)} \mu(\alpha)t^{\deg(\alpha)}. 
$$
Then Proposition~\ref{prop:Mobinv0cycles} shows that 
$$
Z(X,t) = M(X,t)^{-1} = \left (\sum_{\alpha\in Z_{0}^{\eff}(X)} \mu(\alpha)t^{\deg(\alpha)}\right )^{-1}. 
$$

To see the product formula for $Z(X,t)$ (Proposition~\ref{prop:HWzetaprod}) from another angle, for each closed point $x$ of $X$, consider the subposet $(\{ax : a\in\N_{0}\},\leq)\subseteq (Z_{0}^{\eff},\leq)$ which is isomorphic to $(\N_{0},\leq)$. Then $(Z_{0}^{\eff}(X),\leq)$ decomposes as a restricted product of posets 
$$
(Z_{0}^{\eff}(X),\leq) \cong \prod_{x\in |X|}\!\!\!'\, (\{ax\},\leq) \cong \prod_{x\in |X|}\!\!\!'\, (\N_{0},\leq). 
$$
By Example~\ref{ex:leqposet}, the M\"{o}bius function for $(\{ax\},\leq)$ is $\mu_{(\{ax\},\leq)} = \delta - \delta_{x}$, where 
$$
\delta_{x}(\alpha) = \begin{cases}
  1, & \alpha = x\\
  0, & \alpha\not = x. 
\end{cases}
$$
Then the M\"{o}bius function for $(Z_{0}^{\eff}(X),\leq)$ decomposes as well: 
$$
\mu_{(Z_{0}^{\eff}(X),\leq)} = \bigotimes_{x\in |X|} \mu_{(\{ax\},\leq)} = \bigotimes_{x\in |X|} (\delta - \delta_{x}). 
$$
On the level of generating series, 
$$
M(X,t) = \sum_{\alpha\in Z_{0}^{\eff}(X)} \mu(\alpha)t^{\deg(\alpha)} = \prod_{x\in |X|} (\delta - \delta_{x})t^{\deg(x)} = \prod_{x\in |X|} (1 - t^{\deg(x)}). 
$$
Applying M\"{o}bius inversion yields the product form of the Hasse--Weil zeta function: 
$$
Z(X,t) = M(X,t)^{-1} = \prod_{x\in |X|} (1 - t^{\deg(x)})^{-1}. 
$$
\end{ex}

\subsection{Decomposition Sets}
\label{sec:decompset}

The construction of the incidence algebra of a poset has been generalized for categories with certain finiteness conditions called {\it M\"{o}bius categories} \cite{ler}. However, the authors in \cite{gkt1} identify situations (e.g.~rooted trees with cuts) best explained by incidence algebras and M\"{o}bius inversion, despite no obvious M\"{o}bius category structure. This suggests there is a further generalization of M\"{o}bius categories which capture the theory of incidence algebras. Indeed, this is the original motivation behind the definition of decomposition spaces in \cite{gkt1}. In this section, we define the simplicial set version of the construction before giving the general definition in Section~\ref{sec:decompsp}. 

Let $\Delta$ be the category of {\it combinatorial simplices}: objects of $\Delta$ are the finite sets $[n] := \{0,1,\ldots,n\}$ for $n\geq 0$ and morphisms are order-preserving functions $[n]\rightarrow [m]$. A {\it simplicial set} is a functor $K : \Delta^{op}\rightarrow\cat{Set}$, or explicitly, a collection of sets $K_{0},K_{1},K_{2},\ldots$ together with {\it face and degeneracy maps} 
$$
K_{0} \stack{3} K_{1} \stack{5} K_{2} \stack{7} \cdots
$$
satisfying certain compatibility conditions. 

\begin{defn}
Let $K : \Delta^{op}\rightarrow\cat{Set}$ be a simplicial set which is locally of finite length (cf.~\cite{gkt2} for a precise definition). The {\bf incidence coalgebra} of $K$ is the free $k$-vector space $C(K)$ on $K_{1}$ with comultiplication 
\begin{align*}
  \Gamma : C(K) &\longrightarrow C(K)\otimes C(K)\\
    f &\longmapsto \sum_{d_{1}\sigma = f} d_{2}\sigma\otimes d_{0}\sigma
\end{align*}
where the sum is over all $\sigma\in K_{2}$ with $d_{1}\sigma = f$, and counit 
\begin{align*}
  \delta : C(K) &\longrightarrow k\\
    f &\longmapsto \begin{cases}
      1, &\text{ if $f$ is degenerate}\\
      0, &\text{ if $f$ is nondegenerate}. 
    \end{cases}
\end{align*}
The {\bf incidence algebra} of $K$ is the dual $I(K) = \Hom_{k}(C(K),k)$, equipped with multiplication $m = \Gamma^{*} : I(K)\otimes I(K)\rightarrow I(K)$ and unit $\delta$. 
\end{defn}

As with the algebras of arithmetic functions in Section~\ref{sec:zetafunctions}, multiplication in $I(K)$ is a convolution product: 
$$
\varphi*\psi : f \longmapsto \sum_{d_{1}\sigma = f} \varphi(d_{2}\sigma)\otimes\psi(d_{0}\sigma). 
$$
Let $\zeta\in I(K)$ be the zeta function of $K$, sending $f\mapsto 1$ for all $f\in X_{1}$. 

\begin{ex}
When $K = \mathcal{N}(\D)$ is the nerve of a M\"{o}bius category $\mathcal{D}$, $I(\D) := I(\mathcal{N}(\D))$ agrees with the incidence algebra of $\D$ as defined in \cite{ler}. In particular, the convolution product on $I(\D)$ is given by 
$$
\varphi*\psi : f \longmapsto \sum_{h\circ g = f} \varphi(g)\otimes\psi(h)
$$
for any morphism $f\in\mor(\D)$, where the sum is over all factorizations of $f$ in $\mor(\D)$. The axioms of a M\"{o}bius category ensure such a sum is well-defined. 
\end{ex}

In general, $I(K)$ need not be associative or unital. However, these properties hold when $K$ is a decomposition set. To state the definition, we first need the following notions. 

\begin{defn}
In the category $\Delta$, a morphism $g : [m]\rightarrow [n]$ is {\bf active} if $g(0) = 0$ and $g(m) = n$. On the other hand, $g$ is {\bf inert} if $g(i + 1) = g(i) + 1$ for all $0\leq i\leq m - 1$. 
\end{defn}

In other words, active morphisms ``preserve endpoints'' and inert morphisms ``preserve distances''. 

\begin{defn}
A {\bf decomposition set} is a simplicial set $K : \Delta^{op}\rightarrow\cat{Set}$ that takes any pushout diagram in $\Delta$ of the form 
\begin{center}
\begin{tikzpicture}[scale=2]
  \node at (0,1) (a) {$[p]$};
  \node at (1,1) (b) {$[m]$};
  \node at (0,0) (c) {$[\ell]$};
  \node at (1,0) (d) {$[n]$};
  \draw[->] (b) -- (a);
  \draw[->] (c) -- (a);
  \draw[->] (d) -- (b) node[right,pos=.5] {$f$};
  \draw[->] (d) -- (c) node[above,pos=.5] {$g$};
  \draw (.2,.7) -- (.3,.7) -- (.3,.8);
\end{tikzpicture}
\end{center}
where $f$ is inert and $g$ is active, to a pullback diagram in $\cat{Set}$: 
\begin{center}
\begin{tikzpicture}[scale=2]
  \node at (0,1) (a) {$K_{p}$};
  \node at (1,1) (b) {$K_{m}$};
  \node at (0,0) (c) {$K_{\ell}$};
  \node at (1,0) (d) {$K_{n}$};
  \draw[->] (a) -- (b);
  \draw[->] (a) -- (c);
  \draw[->] (b) -- (d) node[right,pos=.5] {$f^{*}$};
  \draw[->] (c) -- (d) node[above,pos=.5] {$g^{*}$};
  \draw (.2,.7) -- (.3,.7) -- (.3,.8);
\end{tikzpicture}
\end{center}
\end{defn}

\begin{ex}
For any category $\mathcal{C}$, the nerve $\mathcal{N}(\mathcal{C})$ is a decomposition set. This follows from the more general statement for decomposition spaces \cite[Prop.~3.7]{gkt1}. 
\end{ex}

\begin{ex}
Here is a concrete explanation of this class of examples, illustrated by the division poset $(\N,\mid)$. In the simplicial nerve $K$ of this poset, a $0$-simplex is a number $a\in\N$, a $1$-simplex is an interval $[a,b]$ and for every $n\geq 2$, a $n$-simplex is a sequence of intervals $[a_{0},a_{1}],[a_{1},a_{2}],\ldots,[a_{n - 1},a_{n}]$. We can visualize an $n$-simplex as a length $n$ composition $a_{0}\rightarrow a_{1}\rightarrow a_{2}\rightarrow\cdots\rightarrow a_{n - 1}\rightarrow a_{n}$, with each arrow representing an interval, i.e.~a divisibility $a_{i}\mid a_{i + 1}$. If $g : [n]\rightarrow [\ell]$ is an active arrow in $\Delta$, then $g^{*} : K_{\ell}\rightarrow K_{n}$ sends a length $\ell$ sequence $a_{0}\rightarrow a_{1}\rightarrow\cdots\rightarrow a_{\ell}$ to the length $n$ sequence 
$$
a_{0} = a_{g(0)}\rightarrow a_{g(1)}\rightarrow\cdots\rightarrow a_{g(n)} = a_{\ell}. 
$$
In other words, the interior factors $a_{g(1)},\ldots,a_{g(n - 1)}$ are a possibly new way of decomposing the divisibility $a_{0}\mid a_{\ell}$. On the other hand, if $f : [n]\rightarrow [m]$ is an inert arrow in $\Delta$, then $f^{*} : K_{m}\rightarrow K_{n}$ sends a length $m$ sequence $b_{0}\rightarrow b_{1}\rightarrow\cdots\rightarrow b_{m}$ to the length $n$ sequence 
$$
b_{f(0)}\rightarrow b_{f(1)} = b_{f(0) + 1}\rightarrow\cdots\rightarrow b_{f(n)} = b_{f(0) + n}. 
$$
This just records a particular divisibility $b_{i}\mid b_{i + n}$ which was part of the original $b_{0}\mid b_{m}$. 

Let's take $n = 1$ for simplicity. If $[p]$ is the pushout of $[\ell]$ and $[m]$ along $f$ and $g$ as in the definition above, then $p = \ell + m - 1$. To see that $K_{\ell + m - 1}$ is the pullback of $K_{\ell}$ and $K_{m}$ along $g^{*}$ and $f^{*}$, notice that for sequences $a_{0}\rightarrow\cdots\rightarrow a_{\ell}\in K_{\ell}$ and $b_{0}\rightarrow\cdots\rightarrow b_{m}\in K_{m}$ to map to the same sequence $c_{0}\rightarrow c_{1}\in K_{1}$, we must have $a_{0} = c_{0} = b_{f(0)}$ and $a_{\ell} = c_{1} = b_{f(0) + 1}$. This determines a sequence 
$$
a_{0}\rightarrow a_{\ell} = b_{f(0) + 1}\rightarrow b_{f(0) + 2}\rightarrow\cdots\rightarrow b_{f(0) + m}\in K_{\ell + m - 1}. 
$$
Conversely, any such sequence in $K_{\ell + m - 1}$ mapping to the same $c_{0}\rightarrow c_{1}\in K_{1}$ along $g^{*}$ and $f^{*}$ must split as above. Therefore $K_{\ell + m - 1}$ is the correct pullback. In plain language, the inert map picks out a certain subsequence to preserve and the active map prescribes a further decomposition of the first term of that subsequence. 
\end{ex}

\begin{prop}
For a decomposition set $K$, the incidence (co)algebra of $K$ is (co)associative and (co)unital. 
\end{prop}

\begin{proof}
This too follows from a general result for decomposition spaces \cite[Sec.~5.3]{gkt1}. %However, a more down-to-earth explanation for decomposition set can be found in \cite[15.3]{abd}. 
\end{proof}

As the authors in \cite{gkt1} explain, a decomposition set is precisely the right structure to be able to define a (co)associative, (co)unital incidence (co)algebra. Most of our important examples so far -- $(\N,\mid)$, $(I_{K}^{+},\mid)$, $(Z_{0}^{\eff}(X),\leq)$ -- fall under the umbrella of decomposition sets and we have seen that the zeta functions in those classical situations all arise from the canonical zeta element in the incidence algebra of the corresponding decomposition set. Nevertheless, two important situations do not admit obvious interpretations using decomposition sets: $L$-functions and motivic zeta functions. 

$L$-functions, for their part, already show up in the incidence algebra of the poset $(\N,\mid)$ by virtue of being Dirichlet series. However, as their coefficients tend to be algebraic numbers, they are not directly amenable to the objective techniques described in Section~\ref{sec:decompsp}; see also \cite{ak1}. Additionally, it is common to regard $L$-functions as a sort of ``twisted'' zeta function, so it is natural to ask for a suitable incidence algebra in which $L(V,s)$ is itself the zeta element, but it is not obvious which decomposition sets give rise to such incidence algebras, if any. In forthcoming work with Jon Aycock, we propose a solution to this problem in the category of simplicial $G$-representations, giving rise to an {\it objective $L$-functor} $L(V)$ for any Galois representation $V$. 

Likewise, motivic zeta functions are out of reach: there is no clear candidate for a decomposition set, not to mention a locally finite poset, that naturally produces the coefficients of $Z_{mot}(X,t)$. Instead, it is natural to replace the category of simplicial sets with the category of simplicial schemes and ask for a suitable analogue of the incidence algebra to house the motivic zeta function. For a partial solution to this problem, see \cite{dh}. 

\subsection{Decomposition Spaces}
\label{sec:decompsp}

In each of the last two situations, we would like to replace the category $\cat{Set}$ with a suitable category $\S$ of spaces, thus passing from set theory to the realm of homotopy theory. In this section, we will take $\S$ to be either the category of simplicial sets or the category of groupoids in order to illustrate the general theory. In \cite{gkt1}, \cite{gkt2} and \cite{gkt3}, as well as related works, the authors work in the $\infty$-category of $\infty$-groupoids. We elect here to keep things as concrete as possible, while noting that such generalizations are readily available. 

\begin{defn}
A {\bf simplicial space} is a functor $X : \Delta^{op}\rightarrow\S$. 
\end{defn}

That is, a simplical space is a collection of spaces $X_{0},X_{1},X_{2},\ldots$ together with {\it face and degeneracy maps} 
$$
X_{0} \stack{3} X_{1} \stack{5} X_{2} \stack{7} \cdots
$$
satisfying certain compatibility conditions. We denote by $s\S$ the category of simplicial spaces. Note that $s\S$ has all limits and colimits and they are computed levelwise. When $\S$ is the category of simplicial sets or the category of groupoids, a {\it discrete simplicial space} is a simplicial space that lies in the essential image of the functor $s\cat{Set}\rightarrow s\S$ induced by the embedding $\cat{Set}\hookrightarrow\S$. 

\begin{defn}
A {\bf decomposition space} is a simplicial space $X : \Delta^{op}\rightarrow\S$ that takes any pushout diagram in $\Delta$ of the form 
\begin{center}
\begin{tikzpicture}[scale=2]
  \node at (0,1) (a) {$[p]$};
  \node at (1,1) (b) {$[m]$};
  \node at (0,0) (c) {$[\ell]$};
  \node at (1,0) (d) {$[n]$};
  \draw[->] (b) -- (a);
  \draw[->] (c) -- (a);
  \draw[->] (d) -- (b) node[right,pos=.5] {$f$};
  \draw[->] (d) -- (c) node[above,pos=.5] {$g$};
  \draw (.2,.7) -- (.3,.7) -- (.3,.8);
\end{tikzpicture}
\end{center}
where $f$ is inert and $g$ is active, to a homotopy pullback diagram in $\S$: 
\begin{center}
\begin{tikzpicture}[scale=2]
  \node at (0,1) (a) {$X_{p}$};
  \node at (1,1) (b) {$X_{m}$};
  \node at (0,0) (c) {$X_{\ell}$};
  \node at (1,0) (d) {$X_{n}$};
  \draw[->] (a) -- (b);
  \draw[->] (a) -- (c);
  \draw[->] (b) -- (d) node[right,pos=.5] {$f^{*}$};
  \draw[->] (c) -- (d) node[above,pos=.5] {$g^{*}$};
  \draw (.2,.7) -- (.3,.7) -- (.3,.8);
\end{tikzpicture}
\end{center}
\end{defn}

We will see that decomposition spaces are in a sense a homotopy-theoretic version of incidence coalgebras. This insight allows us to generalize the algebras of arithmetic functions from Section~\ref{sec:zetafunctions}. 

\begin{ex}
An important example of decomposition spaces is the notion of a {\it Segal space}, due to Rezk \cite{rez} and based on earlier work of Segal \cite{seg}, which generalizes the nerve of a small category in the following way. Let $\mathcal{C}$ be a small category and consider its nerve $\mathcal{N}(\mathcal{C})$ as a simplicial set whose set of $n$th simplices $\mathcal{N}(\mathcal{C})_{n}$ is the set of all strings of $n$ composable morphisms in $\mathcal{C}$ (with obvious face and degeneracy maps). This is an example of a {\it Segal set}, or a simplicial set $K$ such that for every $n\geq 1$, the so-called {\it Segal maps} 
$$
\varphi_{n} : K_{n} \longrightarrow \underbrace{K_{1}\times_{K_{0}}\cdots\times_{K_{0}}K_{1}}_{n}
$$
are bijections. It is an easy consequence of the definition (cf.~\cite[4.4]{rez}) that a simplicial set is a Segal set if and only if it is isomorphic to the nerve of a small category. Rezk upgrades this definition to the context of simplicial spaces by specifying maps of spaces 
$$
\varphi_{n} : K_{n} \longrightarrow \underbrace{K_{1}\times_{K_{0}}^{h}\cdots\times_{K_{0}}^{h}K_{1}}_{n}
$$
where $\times^{h}$ denotes homotopy pullback, and defining a {\it Segal space} to be a simplicial space for which these maps are weak equivalences for all $n\geq 1$. Thus the nerve of a small category is nothing more than a discrete Segal space. Moreover, every Segal space is a decomposition space \cite[Prop.~3.7]{gkt1}. 

In fact, decomposition spaces are precisely the same as Dyckerhoff and Kapranov's notion of {\it $2$-Segal spaces} \cite{dk}, a further generalization of Segal sets. By \cite[Rem.~3.2]{gkt1}, a decomposition space is the same thing as a unital $2$-Segal space, but the unital condition was later shown to be redundant in \cite{fgk}. 
\end{ex}

\subsection{Homotopy Linear Algebra}

As we remarked above, decomposition spaces are a vast generalization of incidence coalgebras. To make this precise, we introduce the reader to the formalism of objective linear algebra (appearing in \cite{lm}) and homotopy linear algebra (as developed in \cite{gkt-hla}). This is a necessary abstraction because the notion of ``free vector space on $1$-simplices'' no longer makes sense in the category of simplicial spaces. Loosely, the idea is to replace vectors and linear maps with spaces and linear functors. In this setting, it is possible to define the incidence coalgebra of a decomposition space and take its homotopy linear algebraic dual to get an incidence algebra. %For a more detailed discussion, see \cite[Sec.~19]{abd}. 

The first step, called {\it objective linear algebra}, is to relax our notions of linear algebra a bit. We take the category $\cat{Set}$ to be our `ground field of scalars', together with the rudimentary operations of addition $S + T = S\amalg T$ and multiplication $ST = S\times T$. Notice that taking cardinality recovers ordinary addition and multiplication on our ordinary scalars, but things like subtraction and inverses, when they are defined, need not lift to the realm of sets. In any case, treating $\cat{Set}$ as the ground field recovers enough aspects of linear algebra to be of use. 

A vector can be represented by a map of sets $v : V\rightarrow S$: the `components' of $v$ are the sets $v^{-1}(s)$ for $s\in S$. Thinking of $S$ as a basis for some vector space (more on this in a moment), the component $v^{-1}(s)$ represents not just how many copies of $s$ are in the vector, but \emph{how they are indexed}. Taking cardinality (of finite sets) recovers our more familiar notion of a vector. Scalar multiplication then is taking a product $A\times V\rightarrow V\rightarrow S$. The sum of two vectors $v : V\rightarrow S$ and $w : W\rightarrow S$ is the vector $v + w : V\amalg W\rightarrow S$ given by the universal property of $\amalg$: 
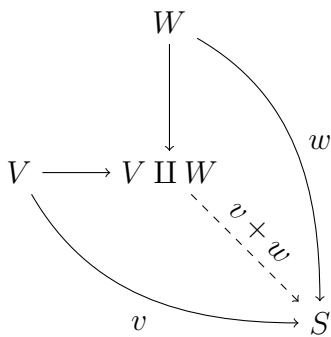
\begin{center}
\begin{tikzpicture}[scale=2]
  \node at (1,1) (b) {$W$};
  \node at (0,0) (c) {$V$};
  \node at (1,0) (d) {$V\amalg W$};
  \node at (2,-1) (e) {$S$};
  \draw[->] (b) -- (d);
  \draw[->] (c) -- (d);
  \draw[->] (b) to[out=-30,in=90] (e);
    \node at (2,.2) {$w$};
  \draw[->] (c) to[out=-60,in=180] (e);
    \node at (.8,-1) {$v$};
  \draw[->,dashed] (d) -- (e);
    \node[rotate=-45] at (1.6,-.4) {$v + w$};
\end{tikzpicture}
\end{center}
Thus the slice category $\cat{Set}_{/S}$ of sets over $S$ (that is, maps $v : V\rightarrow S$) should be regarded as `the vector space with basis $S$'. For this reason, objective linear algebra is sometimes referred to as ``linear algebra with sets''. 

We can also translate linear maps between vector spaces to the objective setting. Suppose for the moment we are dealing with two `finite dimensional' objective vector spaces: slice categories $\cat{Set}_{/S}$ and $\cat{Set}_{/T}$ where $|S| = n < \infty$ and $|T| = m < \infty$. A linear map $\cat{Set}_{/S}\rightarrow\cat{Set}_{/T}$ should then be an analogue of an $m\times n$ matrix of scalars. This can be represented as a map $M\rightarrow S\times T$, which in turn is the same thing as a span $S\leftarrow M\rightarrow T$. One can check %(or cf.~\cite[Sec.~19]{abd}) 
that the usual operations on matrices, including scalar multiplication, addition and matrix-vector and matrix-matrix multiplication, are encoded by span composition. Here, a scalar (a set) is treated as a span $*\leftarrow S\rightarrow *$ and a vector is viewed as either $*\leftarrow V\rightarrow S$ (an $n\times 1$ matrix, if $|S| = n < \infty$) or $S\leftarrow V\rightarrow *$ (a $1\times n$ matrix), where appropriate. 

So a linear map should correspond to a span $S\leftarrow M\rightarrow T$, but we'd like for such a map to actually be a functor $\cat{Set}_{/S}\rightarrow\cat{Set}_{/T}$. Given a `matrix' $S\xleftarrow{f}M\xrightarrow{g}T$ and a `vector' $V\xrightarrow{v}S$, applying the linear map to the vector is encoded by the composition of spans 
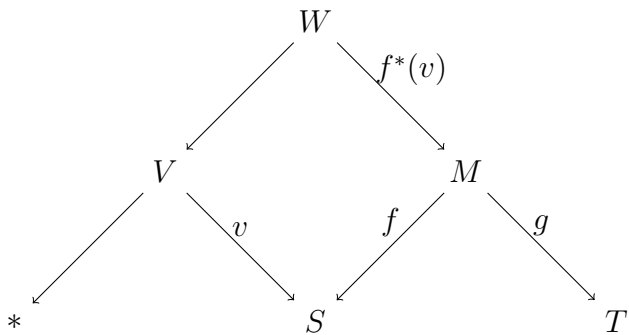
\begin{center}
\begin{tikzpicture}[scale=1.7]
  \node at (-2,0) (0) {$*$};
  \node at (-1,1) (V) {$V$};
  \node at (0,0) (S) {$S$};
  \node at (0,2) (W) {$W$};
  \node at (1,1) (M) {$M$};
  \node at (2,0) (T) {$T$};
  \draw[->] (V) -- (0);
  \draw[->] (V) -- (S) node[above,pos=.5] {$v$};
  \draw[->] (W) -- (V);
  \draw[->] (W) -- (M) node[above,pos=.5] {\phantom{aa} $f^{*}(v)$};
  \draw[->] (M) -- (S) node[above,pos=.5] {$f$};
  \draw[->] (M) -- (T) node[above,pos=.5] {$g$};
\end{tikzpicture}
\end{center}
The output vector is then $W\rightarrow T$, viewed as the larger span $*\leftarrow W\rightarrow T$. More specifically, the map $W\rightarrow T$ is the composition $g_{!}f^{*}(v)$ where $g_{!}$ denotes postcomposition with $g$ and $f^{*}$ denotes the pullback along $f$ in the upper diamond (which is a pullback square). Since $g_{!}$ and $f^{*}$ extend to the slice categories, $g_{!} : \cat{Set}_{/M}\rightarrow\cat{Set}_{/T}$ and $f^{*} : \cat{Set}_{/S}\rightarrow\cat{Set}_{/M}$, it makes sense to take this as a \emph{definition} of a linear map. We will call a functor $a : \cat{Set}_{/S}\rightarrow\cat{Set}_{/T}$ a {\it linear functor} if it factors as $a = g_{!}f^{*}$ for some span $S\xleftarrow{f}M\xrightarrow{g}T$. 

Other operations on vector spaces can be defined in this context as well. The tensor product of two vector spaces $\cat{Set}_{/S}$ and $\cat{Set}_{/T}$ is defined by $\cat{Set}_{/S}\otimes\cat{Set}_{/T} := \cat{Set}_{/S\times T}$. The vector space of linear maps from $S$ to $T$ is the space $\LIN(S,T) := \Fun^{L}(\cat{Set}_{/S},\cat{Set}_{/T})$ of colimit-preserving functors $\cat{Set}_{/S}\rightarrow\cat{Set}_{/T}$ (the superscript $L$ stands for left adjoint, as colimit-preserving functors are the same as left adjoints). Likewise, a vector space dual is given by $(\cat{Set}_{/S})^{*} := \Fun(\cat{Set}_{/S},\cat{Set})$. From the natural equivalence $\cat{Set}_{/S} \simeq \Fun(\cat{Set}_{/S},\cat{Set})$, we recover (cf.~\cite[2.10]{gkt-hla}) the formula 
$$
\LIN(S,T) \simeq (\cat{Set}_{/S\times T})^{*}. 
$$
In particular, $\LIN(S,T)$ is itself an objective vector space and the functors in $\LIN(S,T)$ are given by spans, so they justifiably can be called linear. Plenty more linear algebra can be translated to this objective language, but this suffices for our purposes. 

To promote the above to a {\it homotopy linear algebra}, let $\S$ be the category of spaces. Following \cite{gkt-hla}, we think of $\S$ as our ground field of scalars; a space $S\in\S$ as a basis for the vector space $\S_{/S}$; a morphism $v : V\rightarrow S$, i.e.~an object of $\S_{/S}$, as a vector in the basis $S$; and homotopy products and coproducts as scalar multiplication and addition. Linear maps are a little more delicate to describe. Briefly, the authors in \cite{gkt-hla} construct a category $\LIN$ of ($\infty$-)categories spanned by the slice categories $\S_{/S}$ whose mapping spaces $\LIN(\S_{/S},\S_{/T})$ behave like the spaces of linear functors constructed above. They also construct a tensor product $\S_{/S}\otimes\S_{/T} := \S_{/S\times T}$ and a linear dual $(\S_{/S})^{*} := \Fun(\S_{/S},\S)$ which are also homotopy vector spaces. 

\begin{rem}
As suggested by the parenthetical $\infty$- in the previous paragraph, all of this can be done at the level of $\infty$-categories. Indeed, this is the generality with which the authors in \cite{gkt-hla} state things. Since we do not require the technology of $\infty$-categories in the present article, we leave it to the reader to further explore $\infty$-categorical homotopy linear algebra by reading \cite{gkt-hla}. 
\end{rem}

\subsection{The Incidence Algebra of a Decomposition Space}

Fix a simplical space $X$. 

\begin{defn}
The {\bf incidence coalgebra} of $X$ is the slice category $C(X) := \S_{/X_{1}}$ equipped with linear functors $\Gamma : \S_{/X_{1}}\rightarrow\S_{/X_{1}}\otimes\S_{/X_{1}}$ and $\delta : \S_{/X_{1}}\rightarrow\S$, called {\bf comultiplication} and {\bf counit}, respectively, which are induced by the spans 
$$
\Gamma : X_{1}\xleftarrow{\; d_{1}\;} X_{2}\xrightarrow{(d_{2},d_{0})} X_{1}\times X_{1} \quad\text{and}\quad \delta : X_{1}\xleftarrow{\; s_{0}\;} X_{0}\rightarrow *.
$$
\end{defn}

In the notation above, $\Gamma = (d_{2},d_{0})_{!}d_{1}^{*}$ and $\delta = t_{!}s_{0}^{*}$ where $t : X_{0}\rightarrow *$ is the unique map to the terminal object. 

\begin{prop}[{\cite[Thm.~7.4]{gkt1}}]
If $X$ is a decomposition space, $C(X)$ is a coassociative, counital coalgebra object (homotopy comonoid) in the category $\LIN$, with comultiplication $\Gamma$ and counit $\delta$. 
\end{prop}

Taking the homotopy linear algebraic dual yields a notion of incidence algebra. 

\begin{defn}
The {\bf incidence algebra} of a simplicial space $X$ is the dual $I(X) := (\S_{/X_{1}})^{*} = \Fun(\S_{/X_{1}},\S)$. It is equipped with a linear functor $m : I(X)\otimes I(X)\rightarrow I(X)$ called {\bf multiplication}. Explicitly, for objects $f,g\in I(X)$, their product $m(f,g)$ is given by 
$$
m(f,g) : \S_{/X_{1}}\xrightarrow{\Gamma} \S_{/X_{1}}\otimes\S_{/X_{1}} \xrightarrow{f\otimes g} \S\otimes\S \xrightarrow{\sim} \S. 
$$
\end{defn}

\begin{cor}
If $X$ is a decomposition space, $I(X)$ is an associative, unital algebra object (i.e.~a homotopy monoid) in the category $\LIN$, with multiplication $m$ and unit $\delta$. 
\end{cor}

Every decomposition space $X$ admits a {\it zeta functor} $\zeta\in I(X)$ represented by the span $\zeta : X_{1} \xleftarrow{\operatorname{id}} X_{1} \rightarrow *$. Explicitly, $\zeta$ sends every $1$-simplex to the `scalar' $*$. When $X$ is a decomposition set, this recovers the ordinary zeta function after taking cardinalities on fibres (everything maps to $1$). 

\begin{rem}
\label{rem:pullback}
The utility of homotopy linear algebra becomes clear when we begin comparing the incidence algebras of different decomposition spaces. Let $f : Y\rightarrow X$ be a morphism of simplicial spaces. This induces a map on $1$-simplices, $f_{1} : Y_{1}\rightarrow X_{1}$, which in turn determines a linear functor $f^{*} : I(X)\rightarrow I(Y)$ sending a span $X_{1}\leftarrow V\rightarrow *$ to the composite 
\begin{center}
\begin{tikzpicture}[scale=1.7]
  \node at (-1,1) (Y) {$Y_{1}$};
  \node at (0,0) (X) {$X_{1}$};
  \node at (0,2) (W) {$W$};
  \node at (1,1) (V) {$V$};
  \node at (2,0) (T) {$*$};
  \draw[->] (Y) -- (X);
  \draw[->] (W) -- (Y);
  \draw[->] (W) -- (V);
  \draw[->] (V) -- (X);
  \draw[->] (V) -- (T);
\end{tikzpicture}
\end{center}
In \cite[Sec.~4]{gkt1}, the authors show that when $f$ is a {\it CULF functor}, $f^{*}$ is an algebra homomorphism. An important class of examples of CULF functors are the upper and lower {\it decalage} constructions, which generalize the passage from an incidence algebra of a poset to its reduced subalgebra (see Section~\ref{sec:posets}). In particular, the map 
$$
f : (\N,\mid) \longrightarrow \N^{\times}, \quad [a,b] \longmapsto \frac{b}{a}
$$
where $\N^{\times}$ is the multiplicative monoid of natural numbers, considered as a category with one object, is a CULF map and the induced morphism $f^{*} : I(\N^{\times})\hookrightarrow I(\N,\mid)$ identifies $I(\N^{\times})$ with the reduced incidence subalgebra \cite[Sec.~2.2]{gkt5}. For a number field $K/\Q$, a similar map identifies the reduced incidence subalgebra of $I(I_{K}^{+},\mid)$ with the incidence algebra of the multiplicative monoid of ideals in $\orb_{K}$. 

Meanwhile, for a variety $V/\F_{q}$, let $Z_{0}^{\eff}(V)^{+}$ be the additive monoid of effective $0$-cycles. There is a CULF map 
$$
f : (Z_{0}^{\eff}(V),\leq) \longrightarrow Z_{0}^{\eff}(V)^{+}, \quad [\alpha,\beta] \longmapsto \beta - \alpha. 
$$
Then the image of the induced morphism $f^{*} : I(Z_{0}^{\eff}(V)^{+})\hookrightarrow I(Z_{0}^{\eff}(V),\leq)$ is precisely the reduced subalgebra, in which $\zeta_{V}$ lies. More combinatorial examples can be found in \cite[Sec.~2]{gkt5}. 
\end{rem}

\begin{rem}
\label{rem:pushfwd}
A simplicial map $f : Y\rightarrow X$ induces another map between incidence algebras, this time covariantly. Once again, let $f_{1} : Y_{1}\rightarrow X_{1}$ be the map on $1$-simplices. Then there is a {\it pushforward map} $f_{*} : I(Y)\rightarrow I(X)$ which sends a span $Y_{1}\leftarrow V\rightarrow *$ to the composite $X_{1}\leftarrow Y_{1}\leftarrow V\rightarrow *$. That is, $f_{*}$ ``extends the left leg of every span''. Although $f_{*}$ is rarely an algebra homomorphism, it is still a linear functor and hence has useful applications in the theory of homotopy incidence algebras. For a concrete example, let $K/\Q$ be a number field and let $X = \N^{\times}$ and $Y = I_{K}^{\times}$ be the multiplicative monoids of integral ideals in $\Z$ and $\orb_{K}$, respectively. The norm map $N : I_{K}^{\times}\rightarrow\N^{\times}$ is simplicial, so it induces a pushforward $N_{*} : I(I_{K}^{\times})\rightarrow I(\N^{\times})$. Identify $I(\N^{\times})$ with the reduced subalgebra of $I(\N,\mid)$ as in Remark~\ref{rem:pullback}. Then by Example~\ref{ex:divposet}, after taking cardinalities, $I(\N^{\times})$ is isomorphic to the algebra of Dirichlet series and we can interpret $N_{*}$ as a functor which builds a Dirichlet series for every arithmetic function $f\in A_{K} \cong \widetilde{I}(I_{K}^{+},\mid)$. 

In the case of a variety $V$ over a finite field $k = \F_{q}$, the structure morphism $\pi : V\rightarrow \Spec k$ similarly induces a pushforward $\pi_{*} : \widetilde{I}(Z_{0}^{\eff}(V),\leq) \rightarrow \widetilde{I}(Z_{0}^{\eff}(\Spec k),\leq) \cong \widetilde{I}(\N_{0},\leq)$ which exhibits a power series for every (reduced) arithmetic function on the effective $0$-cycles of $V$. As an example, the zeta functor of $V$ is given by the span 
\begin{center}
\begin{tikzpicture}[scale=1.7]
  \node at (0,0) (X) {$Z_{0}^{\eff}(V)$};
  \node at (1,1) (V) {$Z_{0}^{\eff}(V)$};
  \node at (2,0) (T) {$*$};
  \draw[->] (V) -- (X) node[above,pos=.5] {id};
  \draw[->] (V) -- (T);
\end{tikzpicture}
\end{center}
which decategorifies to the ``numerical'' zeta function $\zeta_{V} : \alpha\mapsto 1$ for each effective $0$-cycle $\alpha$. Applying the pushforward map $\pi_{*}$ produces the span 
\begin{center}
\begin{tikzpicture}[scale=1.7]
  \node at (0,0) (X) {$Z_{0}^{\eff}(\Spec k)$};
  \node at (1,1) (V) {$Z_{0}^{\eff}(V)$};
  \node at (2,0) (T) {$*$};
  \draw[->] (V) -- (X) node[above,pos=.5] {$\pi$};
  \draw[->] (V) -- (T);
\end{tikzpicture}
\end{center}
which in turn decategorifies to the arithmetic function in the Hasse--Weil zeta function $Z(V,t)$. Explicitly, each $1$-simplex in $Z_{0}^{\eff}(\Spec k)$ is of the form $nx$ where $n\geq 0$ and $x$ is the point $\Spec k$. The $n$th coefficient of $\pi_{*}\zeta_{V}$ is computed by identifying the fibre of $nx$ along $\pi$, 
\begin{center}
\begin{tikzpicture}[scale=1.7]
  \node at (-1,1) (Y) {$\{nx\}$};
  \node at (0,0) (X) {$Z_{0}^{\eff}(\Spec k)$};
  \node at (0,2) (W) {$Z_{n}$};
  \node at (1,1) (V) {$Z_{0}^{\eff}(V)$};
  \node at (2,0) (T) {$*$};
  \draw[right hook ->] (Y) -- (X);
  \draw[->] (W) -- (Y);
  \draw[->] (W) -- (V);
  \draw[->] (V) -- (X) node[above,pos=.5] {$\pi$};
  \draw[->] (V) -- (T);
\end{tikzpicture}
\end{center}
and then computing its cardinality: $\#Z_{n} = \#\{\alpha\in Z_{0}^{\eff}(V)\mid \deg(\alpha) = n\}$. 
\end{rem}

In \cite{ak1}, pullback and pushforward maps have been used to prove an objective version of a well-known formula for the Dedekind zeta function of a quadratic number field. Explicitly, let $K/\Q$ be a quadratic number field with objective zeta function $\zeta\in I(I_{K}^{\times})$. Then $K/\Q$ is cut out by a quadratic Dirichlet character $\chi = \legen{D}{\cdot}$, where $D$ is the discriminant of $K$, and we have: 

\begin{thm}[{\cite[Thm.~1.1]{ak1}}]
\label{thm:ak1}
There exist linear functors $L(\chi)^{+},L(\chi)^{-}\in I(I_{\Q}^{\times}) = I(\N^{\times})$ and an equivalence of linear functors 
$$
N_{*}\zeta_{K} + \zeta_{\Q}*L(\chi)^{-} \cong \zeta_{\Q}*L(\chi)^{+}
$$
where $N_{*}$ is the pushforward induced by the field norm map $N = N_{K/\Q}$. 
\end{thm}

Analogously, if $C$ is a hyperelliptic curve over a finite field $k = \F_{q}$, there is a degree $2$ covering map $\pi : C\rightarrow\P_{k}^{1}$. Let $\zeta_{C}\in\widetilde{I}(Z_{0}^{\eff}(X))$ be the objective zeta function of $C$. In \cite{ak2}, we proved the following analogue of Theorem~\ref{thm:ak1}: 

\begin{thm}[{\cite[Thm.~1.1]{ak2}}]
There exist linear functors $L(C)^{+},L(C)^{-}\in\widetilde{I}(Z_{0}^{\eff}(\P^{1}))$ and an equivalence of linear functors 
$$
\pi_{*}\zeta_{C} + \zeta_{\P_{k}^{1}}*L(C)^{-} \cong \zeta_{\P_{k}^{1}}*L(C)^{+}
$$
where $\pi_{*}$ is induced by the double cover $\pi : C\rightarrow\P_{k}^{1}$. 
\end{thm}

To move these objective formulas beyond quadratic extensions and double covers, we plan to utilize objective linear algebra in the category of simplicial $G$-representations, as previewed at the end of Section~\ref{sec:decompset} and below in Section~\ref{sec:future}.

%-------------------------------------------------------------------------------------------------------------------------------------------------------------------------------------

\section{Future Directions}
\label{sec:future}

The full power of decomposition spaces are not needed to describe many of the zeta functions of interest to number theorists, as they arise directly from posets, which are decomposition sets. However, as described in Section~\ref{sec:decompset}, $L$-functions and motivic zeta functions do not fall neatly into the framework of posets. In the latter situation, Das and Howe \cite{dh} construct an incidence algebra for their {\it poscheme of effective $0$-cycles} of a variety and use it to recover the motivic zeta function in the ring $K_{0}(\cat{Var}_{k})[[t]]$ of power series over the Grothendieck ring of $k$-varieties. This construction can also be obtained from a homotopy incidence algebra in the same way as the Hasse--Weil zeta function (Example~\ref{ex:Mobinvvariety}) and we plan to investigate their relationship in future work, which will also give a general description of $L$-functions from this homotopy theory perspective. See also \cite[Appendix A]{ak2} for a brief overview of objective linear algebra for $G$-representations. 

Another type of zeta function that should be amenable to homotopy theoretic methods is the zeta function of an algebraic stack over a finite field. In \cite{beh}, Behrend generalizes the Grothendieck--Lefschetz trace formula to algebraic stacks over a finite field, allowing him to construct the Hasse--Weil zeta function of such a stack. As stacks are presheaves valued in groupoids, homotopy linear algebra is well-suited to the task of encoding the zeta function of a stack using incidence algebras. This investigation will be carried out in future work. 

In their article \cite{cwz}, Campbell, Wolfson and Zakharevich lift the Hasse--Weil zeta function to a map of $K$-theory spectra 
$$
\zeta : K(\cat{Var}_{k}) \longrightarrow K(\Aut(\Z_{\ell}))
$$
where $\Aut(\Z_{\ell})$ denotes the exact category of finitely generated $\Z_{\ell}$-modules with automorphism. They call this the {\it derived $\ell$-adic zeta function} and applying $\pi_{0}$ recovers the Hasse--Weil zeta function via the composition
\begin{align*}
  K_{0}(\cat{Var}_{k}) \xrightarrow{\;\pi_{0}\zeta\;} K_{0}(\Aut(\Z_{\ell})) &\xrightarrow{\;\sim\;} (1 + t\Z_{\ell}[[t]],\cdot)\\
                                                                                 [F] &\longmapsto \det(1 - tF). 
\end{align*}
As the authors suggest in \cite[Question~7.6]{cwz}, one hopes for a lift of the motivic measure 
$$
Z_{mot}(-,t) : K_{0}(\cat{Var}_{k}) \longrightarrow (1 + tK_{0}(\cat{Var}_{k})[[t]],\cdot)
$$
to a map of $K$-theory spectra, ideally in a way that is compatible with the specialization $Z_{mot}(X,t) \mapsto Z(X,t)$ via the motivic measure $\# : K_{0}(\cat{Var}_{k})\rightarrow \Z$. We plan to address this question in future work, using the framework laid out in the present article. More specifically, starting with the simplicial space $\widetilde{S}_{\bullet}(\cat{Var}_{k})$ defined by Campbell in \cite{cam}, one can perform two operations: 
\begin{enumerate}[\quad (a)]
  \item Take its $K$-theory spectrum $K(\cat{Var}_{k})$, as considered in \cite{cam}, \cite{cz} and \cite{cwz}. One might then construct morphisms out of $\widetilde{S}_{\bullet}(\cat{Var}_{k})$ which determine the various maps of ring spectra out of $K(\cat{Var}_{k})$ in \cite{cwz}, especially $\zeta : K(\cat{Var}_{k})\rightarrow K(\Aut(\Z_{\ell}))$. We are currently searching for such a morphism which would give a homotopy theoretic `motivic zeta functor', but at present it is unclear what the target simplicial space should be. 
  
  \item Construct the incidence algebra of $\widetilde{S}_{\bullet}(\cat{Var}_{k})$ and identify its zeta function. One question we have is: in what ways do this abstract zeta function interact with or even determine the Hasse--Weil, derived $\ell$-adic, motivic and other zeta functions? 
\end{enumerate}

Despite not having answers to these questions yet, there is a great deal of information hidden in the structure of $\widetilde{S}_{\bullet}(\cat{Var}_{k})$ and related simplicial objects that can shine a new light on structural aspects of zeta functions.

%-------------------------------------------------------------------------------------------------------------------------------------------------------------------------------------

%-------------------------------------------------------------------------------------------------------------------------------------------------------------------------------------

\end{document}